\documentclass[a4paper,11pt,leqno]{amsart}

\usepackage{amsthm}
\usepackage{amsmath,amssymb,cite,mathrsfs,overpic}
\usepackage{graphicx}
\usepackage{latexsym,bm}
\usepackage{secdot}
\usepackage{amscd}

\DeclareFontEncoding{OT2}{}{}
\DeclareFontSubstitution{OT2}{cmr}{m}{n}
\DeclareSymbolFont{cyss}{OT2}{wncyss}{m}{n}
\DeclareMathSymbol{\sh}{\mathbin}{cyss}{`x}
\DeclareMathOperator*{\Int}{\int \cdots \int}

\newtheorem{theorem}{Theorem}[section]
\newtheorem{lemma}[theorem]{Lemma}
\newtheorem{prop}[theorem]{Proposition}

\theoremstyle{definition}

\newtheorem{example}[theorem]{Example}
\newtheorem{remark}[theorem]{Remark}

\allowdisplaybreaks

\begin{document}

\title[Shuffle products for multiple zeta values]{Shuffle products for multiple zeta values and partial fraction decompositions of zeta-functions of root systems}
\author{Yasushi Komori}
\address{Graduate School of Mathematics, Nagoya University, Chikusa-ku, Nagoya 464-8602 Japan}
\email{komori@math.nagoya-u.ac.jp}

\author{Kohji Matsumoto}
\address{Graduate School of Mathematics, Nagoya University, Chikusa-ku, Nagoya 464-8602 Japan}
\email{kohjimat@math.nagoya-u.ac.jp}

\author{Hirofumi Tsumura}
\address{Department of Mathematics and Information Sciences, Tokyo Metropolitan University, 1-1, Minami-Ohsawa, Hachioji, Tokyo 192-0397 Japan}
\email{tsumura@tmu.ac.jp}

\date{}
\subjclass[2000]{Primary 11M41; Secondary 17B20, 40B05}

\maketitle

\begin{abstract}
The shuffle product plays an important role in the study of multiple zeta values. This is expressed in terms of multiple integrals, and also as a product in a certain non-commutative polynomial algebra over the rationals in two indeterminates. In this paper, we give a new interpretation of the shuffle product. In fact, we prove that the procedure of shuffle products essentially coincides with that of partial fraction decompositions of multiple zeta values of root systems. As an application, we give a proof of extended double shuffle relations without using Drinfel'd integral expressions for multiple zeta values. 
Furthermore, our argument enables us to give some functional relations which include double shuffle relations. 
\end{abstract}

\section{Introduction}\label{sec-1}

Let $\mathbb{N}$, $\mathbb{N}_0$, $\mathbb{Z}$, $\mathbb{Q}$, $\mathbb{R}$ and $\mathbb{C}$ be the set of natural numbers, nonnegative integers, rational integers, rational numbers, real numbers and complex numbers, respectively. 

The multiple zeta value (MZV) of depth $N$ is defined by 
\begin{equation}
\zeta(k_1,k_2,\ldots,k_N)=\sum_{m_1>m_2>\cdots>m_N>0}\frac{1}{m_1^{k_1}m_2^{k_2}\cdots m_N^{k_N}} \label{1-1}
\end{equation}
for $k_1,k_2,\ldots,k_N \in \mathbb{N}$ with $k_1>1$, where $k_1+\cdots+k_N$ is called its weight (see Zagier \cite{Za} and Hoffman \cite{Ho}). Researches on MZVs have advanced actively in this decade (see surveys and related articles \cite{BB0,BB2,Ho3,Kaneko}). 

It is known that the MZV can be expressed as a multiple integral, from which a lot of interesting formulas among MZVs have been systematically obtained. 
Hoffman \cite{Ho2} constructed an algebraic setup for the 
theory of MZVs as follows (see also \cite{BB1}). 
Let $\frak{H}=\mathbb{Q}\langle x,y\rangle$ be the non-commutative polynomial algebra over $\mathbb{Q}$ in indeterminates $x,y$. The monomial $u_1u_2\cdots u_d \in \frak{H}$, where each $u_j=x$ or $y$, is often called a ``word''. Let $\frak{H}^1$ and $\frak{H}^0$ be its subalgebras $\mathbb{Q}+\frak{H}y$ and $\mathbb{Q}+x\frak{H}y$, respectively. Let $Z\,:\,\frak{H}^0 \to \mathbb{R}$ be the $\mathbb{Q}$-linear map defined by 
\begin{equation}
Z\left( x^{k_1-1}yx^{k_2-1}y\cdots x^{k_N-1}y\right)=\zeta\left( k_1,k_2,\cdots,k_N\right). \label{1-5}
\end{equation}
Then it is known that
\begin{equation}
\begin{split}
Z(u_1u_2\cdots u_d)& =\Int_{1>t_1>\cdots>t_d>0}\omega_{u_1}(t_1)\omega_{u_2}(t_2)\cdots \omega_{u_d}(t_d),
\end{split}
\label{1-4}
\end{equation}
where $\omega_x(t)=dt/t$ and $\omega_y(t)=dt/(1-t)$. The right-hand side of \eqref{1-4} is often called the Drinfel'd integral. Note that the connection between the Drinfel'd integrals and MZVs was first noticed by Kontsevich (see \cite{Za}). 
By using this expression, several fascinating relation formulas for MZVs have been obtained, for example, \cite{BBB1,BB2,Ho-Oh,IKZ,Ohno,O-Z} and so on. In particular, it is a key fact that the product of two MZVs can be expressed as a sum of MZVs. Additionally, the associator relations for MZVs were discovered by Drinfel'd \cite{Dri}, and the dimension of the $\mathbb{Q}$-algebra generated by MZVs has been studied (see Goncharov \cite{Gon}, Terasoma \cite{Tera}).

For simplicity, we let $z_p=x^{p-1}y\in \frak{H}^1$ $(p\in \mathbb{N})$. 
It follows that $\{ z_p\,|\,p\in \mathbb{N}\}$ and $1$ are the generators of $\frak{H}^1$. On $\frak{H}^1$, the harmonic product $\ast$ is defined by 
\begin{equation}
\begin{split}
& 1\ast w_1=w_1\ast 1=w_1,\\
& z_p w_1\ast z_q w_2=z_p(w_1\ast z_qw_2)+z_q(z_pw_1\ast w_2)+z_{p+q}(w_1\ast w_2),
\end{split}
\label{1-6}
\end{equation}
for any $p,q \in \mathbb{N}$ and $w_1,w_2\in \frak{H}^1$, and extended by $\mathbb{Q}$-bilinearity. 
Here, $z_p w_1\ast z_q w_2$ is to be read as $(z_p w_1)\ast (z_q w_2)$, but we omit the parentheses. We will use the same convention for $\sh$ products. 
With respect to the harmonic product, it is known that $\frak{H}^1$ becomes a commutative algebra, $\frak{H}^0$ is its subalgebra, and the map $Z$ defined by \eqref{1-4} is an algebra-homomorphism, that is, 
\begin{equation}
Z(w_1\ast w_2)=Z(w_1)Z(w_2)  \label{1-7}
\end{equation}
for $w_1,w_2\in \frak{H}^0$. 

By repeated application of \eqref{1-6}, we find that $w_1\ast w_2$ $(w_1,w_2\in \mathfrak{H}^0)$ can be written as a linear combination of words in $\mathfrak{H}^0$. We call this procedure the ``harmonic product procedure''\,(HPP) for $w_1\ast w_2$. 
For example, we have $z_{p}\ast z_q=z_p z_q+z_qz_p+z_{p+q}$ 
from definition \eqref{1-6}. Hence, by \eqref{1-5} and \eqref{1-7}, we obtain
\begin{equation}
\zeta(p)\zeta(q)=\zeta(p,q)+\zeta(q,p)+\zeta(p+q).  \label{1-7-2}
\end{equation}
Formulas of this type are sometimes called harmonic product relations.

On the other hand, the shuffle product $\sh$ is defined by 
\begin{equation}
\begin{split}
& 1\sh w_1=w_1\sh 1=w_1,\\
& u_1 w_1\sh u_2 w_2=u_1(w_1\sh u_2w_2)+u_2(u_1w_1\sh w_2),
\end{split}
\label{1-8}
\end{equation}
for $w_1,w_2\in \frak{H}^1$ and $u_1,u_2 \in \{x,y\}$, and extended by $\mathbb{Q}$-bilinearity. This product was originally studied by Reutenauer \cite{Reu1,Reu2}. Note that this product can be defined on $\frak{H}$. Similarly to the case of $\ast$, with respect to $\sh$, it is known that $\frak{H}^1$ becomes a commutative algebra, $\frak{H}^0$ is its subalgebra, and the map $Z$ is also an algebra-homomorphism, that is, 
\begin{equation}
Z(w_1\sh w_2)=Z(w_1)Z(w_2)  \label{1-9}
\end{equation}
for $w_1,w_2\in \frak{H}^0$ (see \cite[Theorem 4.1]{Ho-Oh}). 
Applying \eqref{1-8} repeatedly, we find that $w_1\sh w_2$ $(w_1,w_2\in \mathfrak{H}^0)$ can be written as a linear combination of words in $\mathfrak{H}^0$. We call this procedure the ``shuffle product procedure''\,(SPP) for $w_1\sh w_2$. 
Combining \eqref{1-7} and \eqref{1-9}, we obtain 
\begin{equation}
Z(w_1\ast w_2)=Z(w_1\sh w_2). \label{1-10}
\end{equation}
Combining \eqref{1-10} with HPP for $w_1\ast w_2$ and SPP for $w_1\sh w_2$, we obtain non-trivial relations for MZVs, 
which are called (finite) double shuffle relations. 

In this paper, we give a new interpretation of SPP. From definition \eqref{1-8}, for example, we can obtain
\begin{align}
z_p\sh z_rz_q& =\sum_{i=0}^{r-1}\binom{p-1+i}{i}z_{p+i}z_{r-i}z_{q} +\sum_{i=0}^{p-1}\binom{r-1+i}{i} z_{r+i}(z_{p-i}\sh z_q) \label{1-11}
\end{align}
for $p,q,r\in \mathbb{Z}$ with $p,r\geq 2$. 
On the other hand, we will later (see Section \ref{sec-3}) prove that 
\begin{equation}
\begin{split}
\zeta(p)\zeta(r,q)& = \sum_{i=0}^{r-1}\binom{p-1+i}{i}\sum_{l,m,n=1}^\infty \frac{1}{(l+m+n)^{p+i}(m+n)^{r-i}m^q} \\
& +\sum_{i=0}^{p-1}\binom{r-1+i}{i}\sum_{l,m,n=1}^\infty \frac{1}{(l+m+n)^{r+i}l^{p-i}m^q},
\end{split}
\label{1-12}
\end{equation}
by considering the partial fraction decomposition. 
The left-hand side of \eqref{1-11} corresponds to that of \eqref{1-12} via the $Z$-map (by \eqref{1-5}, \eqref{1-9}). Also, the first sum on the right-hand side of \eqref{1-11} corresponds to that on the right-hand side of \eqref{1-12} via the $Z$-map.
So the second sum on the right-hand side of \eqref{1-11} corresponds to that on the right-hand side of \eqref{1-12} via the $Z$-map, and hence the following correspondence is plausible:
\begin{equation}
Z(z_{r+i}(z_{p-i}\sh z_q))=\sum_{l,m,n=1}^\infty \frac{1}{(l+m+n)^{r+i}l^{p-i}m^q}.\label{1-13}
\end{equation}
From this observation, we expect that there is an explicit correspondence between SPP and partial fraction decompositions.

The main aim of this paper is to confirm this expectation. 
We will later confirm \eqref{1-13} and will point out that the right-hand side of \eqref{1-13} is equal to a special value of the zeta-function of the root system of $A_3$ type (see Section \ref{sec-2}) defined by the authors in \cite{KMT1,KMT2,KMT3,MT1}. 
As a generalization of this fact, we will 
prove that each step of SPP can be realized as the corresponding step of partial fraction decompositions of zeta-functions of root systems of $A_N$ type. 

\ 

In Section \ref{sec-2}, we will briefly recall the definitions and results about zeta-functions of root systems of $A_N$ type, and state the main results. More precisely, we will explicitly give the form of each step of SPP (see Proposition \ref{P-2-2}), and will show that each step of SPP can be realized as the corresponding step of partial fraction decompositions of zeta values of root systems (see Theorem \ref{T-2-3}). 
An important point is that, in our argument in the following sections, 
we use neither Drinfel'd integral expressions \eqref{1-4} for MZVs,
nor the multiplicative property \eqref{1-9} a priori.     In particular,
our argument includes a proof of \eqref{1-9} without using Drinfel'd
integral expressions. 
Consequently we give a proof of extended double shuffle relations (see \cite{IKZ}) without using Drinfel'd integral expressions (see Theorem \ref{T-2-5}). 

In Section \ref{sec-3}, 
we will observe the cases of double and triple zeta values, and consider a certain correspondence between SPP and partial fraction decompositions. 

In Section \ref{sec-4}, we will give the proofs of the results stated in Section \ref{sec-2}. 

The fact that SPP can be 
realized by partial fraction decompositions has an important 
application. In fact, in Section \ref{sec-5}, we will give some functional relations which include double shuffle relations. Several years ago, the second-named author proposed the question whether known relations for MZVs are valid only at integer points, or valid also continuously at other points (see \cite{MatXi}). 
Harmonic product relations, such as \eqref{1-7-2}, are valid not only at integer points but also at any other complex points. This is clear because harmonic product relations can be obtained by decompositions of infinite sums. 
Other answers to the above question are given by several explicit functional relations proved recently by the authors (see \cite{KMT2,MT1,MT2,TsCa}), and 
Bradley's partition identities (see \cite{Brad}). 
Now in this paper, we establish that SPP can be described in terms of decompositions of infinite sums. 
Therefore, as in the case of harmonic product relation, we can extend double shuffle relations to functional relations which are valid at any complex points (see Theorem \ref{T-5-1}). 


In Section \ref{sec-6}, we remark that each step of SPP in the sense of \eqref{1-8} can be realized as the corresponding step of partial fraction decompositions of zeta values of root systems.

\ 

\section{Main results} \label{sec-2}

First we briefly recall zeta-functions of root systems of $A_N$ type (for details, see \cite{MT1}, and also \cite{KMT1,KMT2,KMT3}). 

For the complex semisimple Lie algebra $\mathfrak{sl}(N+1)$ of $A_N$ type, 
the zeta-function of the root system of $A_N$ type is defined by 
\begin{equation}
\begin{split}
& \zeta_{N}({\bf s};A_N) =\zeta_{N}({\bf s};\mathfrak{sl}(N+1)) =\sum_{l_1,l_2,\cdots,l_N=1}^\infty \ \prod_{i=1}^{N} \left\{\prod_{j=i}^{N}\left(\sum_{k=i}^{N+i-j}l_{k}\right)^{-s_{ij}}\right\},
\end{split}
\label{2-1}
\end{equation}
where ${\bf s}=(s_{ij})_{1\leq i\leq j\leq N}\in \mathbb{C}^{(N+1)N/2}$ (see \cite[(2.2)]{MT1}). Note that MZV is a special value of this function. In fact, 
from \eqref{2-1}, we see that 
\begin{equation}
\zeta_{N}(s_{11},s_{12},\ldots,s_{1N},0,0,\ldots,0;A_N)=\zeta(s_{11},s_{12},\ldots,s_{1N}). \label{2-2}
\end{equation}
We give the following explicit examples of \eqref{2-1}: 
\begin{equation}
\begin{split}
& \zeta_{1}(s_1;A_1) =\zeta(s_1) =\sum_{n=1}^\infty \frac{1}{n^{s_1}},\\
& \zeta_{2}(s_{11},s_{12},s_{22};A_2) =\sum_{m,n=1}^\infty \frac{1}{(m+n)^{s_{11}}m^{s_{12}} n^{s_{22}}}, \\
& \zeta_{3}(s_{11},s_{12},s_{13},s_{22},s_{23},s_{33};A_3) \\
& \quad =\sum_{l,m,n=1}^\infty \frac{1}{(l+m+n)^{s_{11}}(l+m)^{s_{12}}l^{s_{13}}(m+n)^{s_{22}}m^{s_{23}}n^{s_{33}}},\\
& \zeta_{4}(s_{11},s_{12},s_{13},s_{14},s_{22},s_{23},s_{24},s_{33},s_{34},s_{44};A_4) \\
& \quad =\sum_{l,m,n,h=1}^\infty \frac{1}{(l+m+n+h)^{s_{11}}(l+m+n)^{s_{12}}(l+m)^{s_{13}}l^{s_{14}}} \\
& \quad\quad\quad \times \frac{1}{(m+n+h)^{s_{22}}(m+n)^{s_{23}}m^{s_{24}}(n+h)^{s_{33}}n^{s_{34}}h^{s_{44}}}.
\end{split}
\label{2-3}
\end{equation}
Note that the orders of variables $\{s_{ij}\}$ on the left-hand sides are different from those in our previous papers \cite{KMT1,KMT2,MT1}. Here, we adopt the order of variables corresponding to that of MZVs.

\begin{remark} \label{R-2-1}
In our previous work, we defined zeta-functions of root systems of not only $A_N$ type but also any $X_r$ type which we denote by $\zeta_N({\bf s};X_N)$, where $X=A,B,C,D,E,F,G$ (for details, see \cite{KMT1,KMT2,KMT3}). 
The origin of these functions goes back to Witten zeta-functions which were studied by Witten \cite{Wi} in connection with quantum gauge theory, and essentially formulated by Zagier in \cite{Za}. For any semisimple Lie algebra ${\frak g}$ of $X_r$ type, the Witten zeta-function associated with ${\frak g}$ is a one-variable function defined by the form $K({\frak g})^s \zeta_N(s,s,\ldots,s;X_r)$, where $K({\frak g})$ is a constant depending only on ${\frak g}$ (see \cite[Section 2]{KMT3}). 
In the 1950's, Tornheim \cite{To} first studied the value $\zeta_2(d_1,d_2,d_3;A_2)$ for $d_1,d_2,d_3 \in \mathbb{N}$, which is called the Tornheim double sum. Independently, Mordell \cite{Mo} studied the value $\zeta_2(2d,2d,2d;A_2)$ $(d\in \mathbb{N})$. From the analytic viewpoint, the second-named author \cite{MatMil} studied the multi-variable function $\zeta_2(s_1,s_2,s_3;A_2)$ for $s_1,s_2,s_3 \in \mathbb{C}$ which is called the Mordell-Tornheim double zeta-function, denoted by $\zeta_{MT,2}(s_1,s_2,s_3)$. He showed the meromorphic continuation of $\zeta_{MT,2}(s_1,s_2,s_3)$, etc. Note that the function $2^s\zeta_{MT,2}(s,s,s)$ coincides with the Witten zeta-function $K(\mathfrak{sl}(3))^s \zeta_2(s,s,s;A_2)$ associated with the Lie algebra $\frak{sl}(3)$ (see Witten \cite{Wi}, Zagier \cite{Za}). In \cite{MatBonn}, the second-named author considered $\zeta_{2}(s_1,s_2,s_3,s_4;B_2)$ which can be regarded as a multi-variable version of the Witten zeta-function of $B_2$ type. 
As a generalization of these results, we generally defined multi-variable zeta-functions $\zeta_N({\bf s};X_N)$ of the root system of $X_r$ type and studied their properties in \cite{KMT1,KMT2,KMT3,MT1}.
\end{remark}

To state our main results, we first prepare the following proposition, which is a generalization of \eqref{1-11} or \eqref{3-1} below. 
Note that the empty sum and the empty product are to be understood as $0$ and $1$, respectively.

\begin{prop}\label{P-2-2}
For $a,b \in \mathbb{N}$ and $p_1,\ldots,p_a,\,q_1,\ldots,q_b \in \mathbb{N}$,
\begin{equation}
\begin{split}
& z_{p_a}\cdots z_{p_1}\sh z_{q_b}\cdots z_{q_1} \\
& =\sum_{\tau=0}^{q_b-1}\binom{p_a-1+\tau}{\tau}z_{p_a+\tau}(z_{p_{a-1}}\cdots z_{p_1} \sh z_{q_b-\tau}z_{q_{b-1}}\cdots z_{q_1}) \\
& \quad +\sum_{\tau=0}^{p_a-1}\binom{q_b-1+\tau}{\tau} z_{q_b+\tau}(z_{p_a-\tau}z_{p_{a-1}}\cdots z_{p_1}\sh z_{q_{b-1}}\cdots z_{q_1}).
\end{split}
\label{2-4}
\end{equation}
\end{prop}

Consequently, SPP is performed by applying this proposition repeatedly. 
Then we can see that only terms of the form 
$$z_{r_{c}}z_{r_{c-1}}\cdots z_{r_1}(z_{p_a}z_{p_{a-1}}\cdots z_{p_1}\sh z_{q_b}z_{q_{b-1}}\cdots z_{q_1})$$
appear in SPP. Hence, from \eqref{2-4}, 
we find that each step of SPP is expressed by
\begin{equation}
\begin{split}
& z_{r_{c}}z_{r_{c-1}}\cdots z_{r_1}(z_{p_a}\cdots z_{p_1}\sh z_{q_b}\cdots z_{q_1}) \\
& =\sum_{\tau=0}^{q_b-1}\binom{p_a-1+\tau}{\tau}z_{r_{c}}z_{r_{c-1}}\cdots z_{r_1}z_{p_a+\tau}(z_{p_{a-1}}\cdots z_{p_1} \sh z_{q_b-\tau}z_{q_{b-1}}\cdots z_{q_1}) \\
& \quad +\sum_{\tau=0}^{p_a-1}\binom{q_b-1+\tau}{\tau} z_{r_{c}}z_{r_{c-1}}\cdots z_{r_1}z_{q_b+\tau}(z_{p_a-\tau}z_{p_{a-1}}\cdots z_{p_1}\sh z_{q_{b-1}}\cdots z_{q_1}).
\end{split}
\label{2-4-2}
\end{equation}
For each term on the both sides, we can determine its image via the $Z$-map as follows.

\begin{theorem} \label{T-2-3}
Let $a,b\in \mathbb{N}$ and $c \in \mathbb{N}_{0}$. For $(p_\eta)\in \mathbb{N}^a$, $(q_\xi)\in \mathbb{N}^b$, $(r_\sigma)\in \mathbb{N}^c$, 
\begin{equation}
\begin{split}
& Z\left(z_{r_{c}}z_{r_{c-1}}\cdots z_{r_1}(z_{p_a}z_{p_{a-1}}\cdots z_{p_1}\sh z_{q_b}z_{q_{b-1}}\cdots z_{q_1})\right) =\zeta_{a+b+c}\left(\{d_{ij}\};A_{a+b+c}\right)
\end{split}
\label{2-5}
\end{equation}
holds under the assumption that the right-hand side of \eqref{2-5} is convergent, where 
\begin{equation}
d_{ij}=
\begin{cases}
r_{c+1-j} & (i=1;\,1\leq j \leq c) \\
p_{a+b+c+1-j} & (i=1;\,b+c+1\leq j \leq a+b+c) \\
q_{a+b+c+1-j} & (i=a+1;\,a+c+1 \leq j \leq a+b+c)\\
0   & (\text{otherwise}).
\end{cases}
\label{2-6}
\end{equation}
Furthermore 
each step of the shuffle product procedure (SPP) given by \eqref{2-4-2} can be realized as the corresponding step of partial fraction decompositions (PFD) of zeta values of root systems. 
\end{theorem}

\begin{remark} \label{R-2-4}
Theorem \ref{T-2-3} can be illustrated as the following commutative diagram of procedures:
\begin{equation*}
\begin{CD}
\begin{tabular}{|c|} \hline $\ \ \mathfrak{H}^{0}\ \ $ \\ \hline \end{tabular} @>\text{$Z$}>> \begin{tabular}{|c|} \hline {\rm Zeta values of root systems} \\ \hline \end{tabular} \\
@V\text{\rm SPP}VV @VV\text{\rm PFD}V \\
\begin{tabular}{|c|} \hline $\ \ \mathfrak{H}^{0}\ \ $ \\ \hline \end{tabular} @>\text{$Z$}>> \begin{tabular}{|c|} \hline {\rm Zeta values of root systems} \\ \hline \end{tabular}  \\
@V\text{\rm SPP}VV @VV\text{\rm PFD}V \\
\end{CD}
\end{equation*}
where SPP implies the step determined by \eqref{2-4-2} and \textrm{PFD} implies that of a certain partial fraction decomposition which corresponds to SPP. The explicit correspondence between SPP and PFD will be given in the proof of Theorem \ref{T-2-3}.
\end{remark}

\begin{remark} \label{R-2-4-2}
In Theorem \ref{T-2-3}, we only consider the case that the right-hand side of \eqref{2-5} is convergent. We can explicitly state the condition of this case as follows: 
$p_a\geq 2,\,q_b\geq 2$ if $c=0$;\ $r_c\geq 2$ if $c\geq 1$.
\end{remark}

\ 

Recall the algebra-homomorphism property \eqref{1-9} of $Z$ with respect to $\sh$: 
\begin{equation*}
Z(w_1\sh w_2)=Z(w_1)Z(w_2)\quad (w_1,w_2\in \mathfrak{H}^{0}).
\end{equation*}
This property has been proved by making use of Drinfel'd integral expressions for MZVs. On the other hand, our argument in Section \ref{sec-4} includes a proof of \eqref{1-9} without using Drinfel'd integral expressions for MZVs. We recall the regularized $Z$-map $Z^\sh\,:\,\mathfrak{H}^1 \to \mathbb{R}[T]$ defined by
\begin{equation}
Z^\sh(w)=
\begin{cases}
T & (w=y),\\
Z(w) & (w\in \mathfrak{H}^{0}),
\end{cases}
\label{2-7}
\end{equation}
where $T$ is an indeterminate (see \cite[Proposition 1]{IKZ}). As for $Z$ and $Z^\sh$, we obtain the following.

\begin{theorem} \label{T-2-5}
The algebra-homomorphism property of the $Z$-map with respect to $\sh$: 
\begin{equation}
Z(w_1\sh w_2)=Z(w_1)Z(w_2)\quad (w_1,w_2\in \mathfrak{H}^0) \label{2-8}
\end{equation}
can be proved without using Drinfel'd integral expressions for MZVs, by only using the partial fraction decompositions. Consequently, the extended double shuffle relation
\begin{equation}
Z^\sh(w_1\ast w_2-w_1\sh w_2)=0\quad (w_1\in \mathfrak{H}^1,\,w_2\in \mathfrak{H}^0)  \label{2-9}
\end{equation}
can be proved without using Drinfel'd integral expressions for MZVs. 
\end{theorem}

\begin{remark} 
In this paper, we realize MZVs as special values of zeta-functions of 
root systems of $A_N$ type.    It is possible that our theory is 
formulated without the terminology of root systems.  However 
this realization gives a new insight in the theory of MZVs, and
has some applications.
For example, since $A_N\subset C_N$ in the sense of 
\cite[Theorem 5.2]{KMT3}, it is possible to realize MZVs as special 
values of zeta-functions of root systems of $C_N$ type, and then, 
MZVs correspond to 
those with $s_\alpha=0$ for all short roots $\alpha$. In particular, 
the well-known result
\begin{equation}
\zeta(2k,\ldots,2k)\in \mathbb{Q}\cdot \pi^{2kN}\quad (k\in \mathbb{N})
\end{equation}
for MZVs of depth $N$ 
can be regarded as a generalization of Witten's volume formula 
for zeta-functions of root systems of $C_N$ type (see \cite[Theorem 4.6]{KMT4}).
\end{remark}

\ 

\section{Double and triple zeta values}  \label{sec-3}

In this section, we discuss double and triple zeta values. The results in this section are included in our main theorem, but we insert this section so that the readers can easily catch the essence of our idea from the description of these simple examples. 

First we consider the double zeta case. From \eqref{1-8}, we see that SPP for $z_p\sh z_q$ gives 
\begin{equation}
z_p \sh z_q =\sum_{\tau=0}^{q-1}\binom{p-1+\tau}{\tau}z_{p+\tau}z_{q-\tau} +\sum_{\tau=0}^{p-1}\binom{q-1+\tau}{\tau}z_{q+\tau}z_{p-\tau}, \label{3-1}
\end{equation}
which is mapped to 
\begin{align}
Z(z_p \sh z_q)& =\sum_{\tau=0}^{q-1}\binom{p-1+\tau}{\tau}\zeta(p+\tau,q-\tau) +\sum_{\tau=0}^{p-1}\binom{q-1+\tau}{\tau}\zeta(q+\tau,p-\tau) \label{3-2} 
\end{align}
via the $Z$-map for $p,q \in \mathbb{Z}$ with $p,q\geq 2$. On the other hand, we know the following partial fraction decomposition:
\begin{equation}
\begin{split}
&\frac{1}{X^\alpha Y^\beta} =\sum_{\tau=0}^{\beta-1}\binom{\alpha-1+\tau}{\tau}\frac{1}{(X+Y)^{\alpha+\tau}Y^{\beta-\tau}}  +\sum_{\tau=0}^{\alpha-1}\binom{\beta-1+\tau}{\tau}\frac{1}{(X+Y)^{\beta+\tau}X^{\alpha-\tau}}
\end{split}
\label{3-3}
\end{equation}
as a relation for formal rational functions, which can be easily proved by induction (see, for example, \cite[(1.6)]{HWZ}, \cite{Mark}). From \eqref{3-3}, we immediately obtain
\begin{align}
\zeta(p)\zeta(q)& =\sum_{\tau=0}^{q-1}\binom{p-1+\tau}{\tau}\zeta(p+\tau,q-\tau) +\sum_{\tau=0}^{p-1}\binom{q-1+\tau}{\tau}\zeta(q+\tau,p-\tau). \label{3-3-2} 
\end{align}
Comparing \eqref{3-2} and \eqref{3-3-2}, we obtain 
\begin{equation}
Z(z_p \sh z_q)=\zeta(p)\zeta(q)=Z(z_p)Z(z_q). \label{3-3-3}
\end{equation}
The above argument implies the assertion in the double case of Theorem \ref{T-2-5} and also implies that the shuffle product \eqref{3-1} corresponds to the partial fraction decomposition \eqref{3-3}. Note that Euler already noticed the correspondence between \eqref{3-3} and \eqref{3-3-2} (see Borwein et al. \cite{BBB0}). 

Next we consider the triple zeta case which is a typical case of our investigation. 
Setting $(X,Y)=(l,m+n)$ and $(\alpha,\beta)=(p,r)$ in the partial fraction
decomposition \eqref{3-3}, multiplying the both sides by $m^{-q}$, and
then summing up with respect to $l,m$ and $n$, we obtain \eqref{1-12}. 
Next we show \eqref{1-13}. In fact, it follows from \eqref{3-1} that
\begin{equation}
\begin{split}
z_{r+i}(z_{p-i}\sh z_q)& =\sum_{\tau=0}^{q-1}\binom{p-i-1+\tau}{\tau}z_{r+i}z_{p-i+\tau}z_{q-\tau} \\
& \quad +\sum_{\tau=0}^{p-i-1}\binom{q-1+\tau}{\tau}z_{r+i}z_{q+\tau}z_{p-i-\tau}.
\end{split}
\label{3-5}
\end{equation}
The right-hand side of \eqref{3-5} is mapped to 
\begin{equation}
\begin{split}
& \sum_{\tau=0}^{q-1}\binom{p-i-1+\tau}{\tau}\zeta(r+i,p-i+\tau,q-\tau) \\
& \quad +\sum_{\tau=0}^{p-i-1}\binom{q-1+\tau}{\tau}\zeta(r+i, q+\tau,p-i-\tau)
\end{split}
\label{3-6} 
\end{equation}
via the $Z$-map. On the other hand, setting $(X,Y)=(l,m)$ and $(\alpha,\beta)=(p-i,q)$ in the partial fraction decomposition \eqref{3-3}, multiplying the both sides by $(l+m+n)^{-r-i}$, and then summing 
up with respect to $l, m$ and $n$, we see that \eqref{3-6} is equal
to the right-hand side of \eqref{1-13}. Therefore we see that \eqref{1-13} indeed holds, that is,
\begin{equation}
\begin{split}
Z(z_{r+i}(z_{p-i}\sh z_q))& =\sum_{l,m,n=1}^\infty \frac{1}{(l+m+n)^{r+i}l^{p-i}m^q}\\
& =\zeta_3(r+i,0,p-i,0,q,0;A_3).
\end{split}
\label{3-7}
\end{equation}
Hence we obtain the case $(a,b,c)=(1,1,1)$ of \eqref{2-5} in Theorem \ref{T-2-3}. Moreover \eqref{3-7} implies that the right-hand side of \eqref{1-11} is mapped to the right-hand side of \eqref{1-12} 
via the $Z$-map. 
Therefore the left-hand side of \eqref{1-11} should be mapped to the
left-hand side of \eqref{1-12} via the $Z$-map, that is
\begin{equation}
\begin{split}
Z(z_p\sh z_rz_q)& =\zeta(p)\zeta(r,q)=Z(z_p)Z(z_rz_q).
\end{split}
\label{3-9-2} 
\end{equation}
In Section \ref{sec-1}, we assume \eqref{3-9-2} as a special case
of \eqref{1-9}, and deduce the correspondence of the left-hand
sides of \eqref{1-11} and \eqref{1-12}.   Here we proceed reversely to
obtain \eqref{3-9-2} as a conclusion.
Hence the triple case in Theorem \ref{T-2-5}.

\begin{remark}
Arakawa and Kaneko \cite{AK} studied double shuffle relations for multiple $L$-values by considering their own algebraic setup. For example, as a shuffle product relation, it is shown that
\begin{equation}
\begin{split}
L(p;a_1)L(q;a_2) =& \sum_{\tau=0}^{q-1}\binom{p-1+\tau}{\tau}L_{\sh}(p+\tau,q-\tau;a_1,a_2) \\
& +\sum_{\tau=0}^{p-1}\binom{q-1+\tau}{\tau}L_{\sh}(q+\tau,p-\tau;a_2,a_1),
\end{split}
\label{3-10}
\end{equation}
where $L(s;a)=\sum_{m\geq 1}e^{2\pi i a/f}m^{-s}$ and 
$$L_{\sh}(s_1,s_2;a_1,a_2)=\sum_{m,n=1}^\infty \frac{e^{2\pi ima_1/f}e^{2\pi ina_2/f}}{(m+n)^{s_1}n^{s_2}}$$
for $f \in \mathbb{N}$ and $a,a_1,a_2\in \mathbb{Z}$ with $0\leq a,a_1,a_2<f$. 
Similarly to the above consideration, it is clear that \eqref{3-10} can be deduced from \eqref{3-3}. The definition of the shuffle product $\sh$ for multiple $L$-values is essentially the same as that for MZVs. Therefore we can see that SPP for multiple $L$-values corresponds to that of partial fraction decompositions of multiple $L$-functions of root systems studied in \cite{KMT2,KMT-L}.
\end{remark}

\ 

\section{Proofs of main results}  \label{sec-4}

In this section, we will give the proofs of results stated in Section \ref{sec-2}.

\begin{proof}[Proof of Proposition \ref{P-2-2}]
We will prove this proposition by induction on $a+b\geq 2$. In the case $a+b=2$, namely $a=b=1$, the assertion implies \eqref{3-1}. 

Next we consider the case $a+b>2$. In this case, we further use the induction on $p_a+q_b\geq 2$. In the case $p_a+q_b=2$, namely $p_a=q_b=1$, we have $z_{p_a}=z_{q_b}=y$. Hence, by \eqref{1-8}, we see that 
\begin{align*}
& z_{p_a}\cdots z_{p_1}\sh z_{q_b}\cdots z_{q_1} \\
& \quad = z_{p_a}(z_{p_{a-1}}\cdots z_{p_1}\sh z_{q_b}\cdots z_{q_1})+z_{q_b}( z_{p_a}\cdots z_{p_1}\sh z_{q_{b-1}}\cdots z_{q_1}), 
\end{align*}
which implies \eqref{2-4} in the case $p_a=q_b=1$. 
Hence we have the assertion.

Now we consider the general case. Here we need to check the following three cases:

\quad (Case\,1) $p_a=1$ and $q_b>1$, 

\quad (Case\,2) $p_a>1$ and $q_b=1$, 

\quad (Case\,3) $p_a>1$ and $q_b>1$. \\
First we consider Case 1. Since $p_a=1$, namely $z_{p_a}=y$, it follows from \eqref{1-8} that 
\begin{equation}
\begin{split}
& z_{p_a}\cdots z_{p_1}\sh z_{q_b}\cdots z_{q_1} \\
& \quad = z_{1}(z_{p_{a-1}}\cdots z_{p_1}\sh z_{q_b}\cdots z_{q_1})+x(z_{1}z_{p_{a-1}}\cdots z_{p_1}\sh z_{q_{b}-1}z_{q_{b-1}}\cdots z_{q_1}). 
\end{split}
\label{4-1}
\end{equation}
By the assumption of induction in the case of $p_a+q_b-1$, the second term on the right-hand side is 
\begin{equation*}
\begin{split}
& x\sum_{\tau=0}^{q_b-2}z_{1+\tau}(z_{p_{a-1}}\cdots z_{p_1}\sh z_{q_b-1-\tau}z_{q_{b-1}}\cdots z_{q_1}) +xz_{q_b-1}(z_{1}z_{p_{a-1}}\cdots z_{p_1}\sh z_{q_{b-1}}\cdots z_{q_1})\\
& =\sum_{\rho=1}^{q_b-1}z_{1+\rho}(z_{p_{a-1}}\cdots z_{p_1}\sh z_{q_b-\rho}z_{q_{b-1}}\cdots z_{q_1}) + z_{q_b}(z_{1}z_{p_{a-1}}\cdots z_{p_1}\sh z_{q_{b-1}}\cdots z_{q_1}),
\end{split}
\end{equation*}
by putting $\rho=\tau+1$ in the first sum on the left-hand side and using $xz_{a}=z_{a+1}$. 
Hence, by \eqref{4-1}, we obtain \eqref{2-4} in Case 1. 
Similarly, we can obtain \eqref{2-4} in Case 2.  

Now we consider Case 3. 
From \eqref{1-8} and the assumption of induction in the case of $p_a+q_b-1$, we have
\begin{equation}
\begin{split}
z_{p_a}& \cdots z_{p_1}\sh z_{q_b}\cdots z_{q_1} \\
& = x(z_{p_a-1}z_{p_{a-1}}\cdots z_{p_1}\sh z_{q_b}\cdots z_{q_1})+x(z_{p_a}\cdots z_{p_1}\sh z_{q_{b}-1}z_{q_{b-1}}\cdots z_{q_1})\\
& =x\sum_{\tau=0}^{q_b-1}\binom{p_a-2+\tau}{\tau}z_{p_a-1+\tau}(z_{p_{a-1}}\cdots z_{p_1}\sh z_{q_b-\tau}z_{q_{b-1}}\cdots z_{q_1}) \\
& \quad +x\sum_{\tau=0}^{p_a-2}\binom{q_b-1+\tau}{\tau} z_{q_b+\tau}(z_{p_{a}-1-\tau}z_{p_{a-1}}\cdots z_{p_1}\sh z_{q_{b-1}}\cdots z_{q_1})\\
& \quad +x\sum_{\tau=0}^{q_b-2}\binom{p_a-1+\tau}{\tau}z_{p_a+\tau}(z_{p_{a-1}}\cdots z_{p_1}\sh z_{q_b-1-\tau}z_{q_{b-1}}\cdots z_{q_1}) \\
& \quad +x\sum_{\tau=0}^{p_a-1}\binom{q_b-2+\tau}{\tau} z_{q_b-1+\tau}(z_{p_{a}-\tau}z_{p_{a-1}}\cdots z_{p_1}\sh z_{q_{b-1}}\cdots z_{q_1}).
\end{split}
\label{4-2}
\end{equation}
Replace $\tau+1$ by $\tau$ in the second and the third sums on the right-hand side of \eqref{4-2}, and use the relation 
$$\binom{m+n-1}{n}+\binom{m+n-1}{n-1}=\binom{m+n}{n}\quad (m,n\in \mathbb{N})$$
in order to unite the first and the third sums, and the second and the fourth sums, respectively. Then, noting $xz_{p}=z_{p+1}$, we obtain \eqref{2-4} in the case of $p_a+q_b$. Thus, by induction, we complete the proof of Proposition \ref{P-2-2}. 
\end{proof}

\begin{proof}[Proof of Theorem \ref{T-2-3}]
For our assertion of the former part of Theorem \ref{T-2-3}, we first prove that 
\begin{equation}
\begin{split}
& \zeta_{a+b+c}\left(\{d_{ij}\};A_{a+b+c}\right)\\
&=\sum_{l_1,l_2,\ldots ,l_a \in \mathbb{N} \atop {m_1,m_{2},\ldots,m_b \in \mathbb{N} \atop n_{1},n_2,\ldots,n_{c}\in \mathbb{N}}} \prod_{\sigma=1}^{c}\left( \sum_{\nu=1}^{\sigma}n_\nu+\sum_{\mu=1}^{a}l_\mu+\sum_{\rho=1}^{b}m_\rho\right)^{-r_\sigma}\prod_{\eta=1}^{a}\left(\sum_{\mu=1}^{\eta}l_\mu\right)^{-p_\eta} \prod_{\xi=1}^{b}\left(\sum_{\rho=1}^{\xi}m_\rho\right)^{-q_\xi},
\end{split}
\label{4-2-1}
\end{equation}
where
\begin{align*}
& d_{1,a+b+c+1-j}=r_{j-a-b} \quad (a+b+1 \leq j\leq a+b+c),\\
& d_{1,a+b+c+1-j}=p_j \quad (1\leq j\leq a),\\
& d_{a+1,a+b+c+1-j}=q_j\quad (1\leq j\leq b),
\end{align*}
and $d_{ij}=0$ otherwise. 
In fact, by replacing $j$ by $N+1-j$ on the right-hand side of \eqref{2-1}, 
we can see that 
\begin{equation}
\begin{split}
\zeta_{N}(\{ s_{ij}\};A_N)=\sum_{l_1,\ldots,l_N}\prod_{i=1}^{N}\left\{ \prod_{j=1}^{N-i+1}\left(\sum_{k=i}^{i+j-1}l_k\right)^{-s_{i,N+1-j}}\right\}.
\end{split}
\label{4-2-2}
\end{equation}
Comparing \eqref{4-2-1} and \eqref{4-2-2} with $N=a+b+c$, we see that $m_\rho$ and $n_{\nu}$ correspond to $l_{a+\rho}$ and $l_{a+b+\nu}$, respectively. Therefore we clearly obtain \eqref{4-2-1}.

Now we prove 
\begin{equation}
\begin{split}
& Z\left(z_{r_{c}}z_{r_{c-1}} \cdots z_{r_1}(z_{p_a}\cdots z_{p_1}\sh z_{q_b}\cdots z_{q_1})\right) \\
&=\sum_{l_1,l_2,\ldots ,l_a \in \mathbb{N} \atop {m_1,m_{2},\ldots,m_b \in \mathbb{N} \atop n_{1},n_2,\ldots,n_{c}\in \mathbb{N}}} \prod_{\sigma=1}^{c}\left( \sum_{\nu=1}^{\sigma}n_\nu+\sum_{\mu=1}^{a}l_\mu+\sum_{\rho=1}^{b}m_\rho\right)^{-r_\sigma}\prod_{\eta=1}^{a}\left(\sum_{\mu=1}^{\eta}l_\mu\right)^{-p_\eta} \prod_{\xi=1}^{b}\left(\sum_{\rho=1}^{\xi}m_\rho\right)^{-q_\xi}
\end{split}
\label{4-5}
\end{equation}
by induction on $a+b\geq 2$ $(a,b \in \mathbb{N})$. We first consider the case $a+b=2$, that is, $(a,b)=(1,1)$. From \eqref{1-5} and \eqref{3-1}, we have
\begin{equation}
\begin{split}
& Z(z_{r_{c}}z_{r_{c-1}} \cdots z_{r_1}(z_{p_1}\sh z_{q_1})) \\
& =\sum_{\tau=0}^{q_1-1}\binom{p_1-1+\tau}{\tau}Z\left(z_{r_{c}}z_{r_{c-1}}\cdots z_{r_1}z_{p_1+\tau}z_{q_1-\tau}\right) \\
& \quad +\sum_{\tau=0}^{p_1-1}\binom{q_1-1+\tau}{\tau}Z\left(z_{r_{c}}z_{r_{c-1}} \cdots z_{r_1}z_{q_1+\tau}z_{p_1-\tau}\right) \\
& =\sum_{\tau=0}^{q_1-1}\binom{p_1-1+\tau}{\tau}\sum_{l, m \atop n_{1},\ldots,n_c} \frac{1}{\left\{\prod_{\xi=1}^{c} \left(\sum_{\nu=1}^{\xi}n_\nu+l+m \right)^{r_\xi}\right\} (l+m)^{p_1+\tau}m^{q_1-\tau}} \\
& \quad +\sum_{\tau=0}^{p_1-1}\binom{q_1-1+\tau}{\tau}\sum_{l,m \atop n_1,\ldots,n_c} \frac{1}{\left\{\prod_{\xi=1}^{c} \left(\sum_{\nu=1}^{\xi}n_\nu+l+m\right)^{r_\xi}\right\} (l+m)^{q_1+\tau} l^{p_1-\tau}}.
\end{split}
\label{4-2-3}
\end{equation}
Using the partial fraction decomposition \eqref{3-3} with $(X,Y)=(l,m)$ and $(\alpha,\beta)=(p_1,q_1)$, we see that the right-hand side of \eqref{4-2-3} coincides with 
$$\sum_{l, m \atop n_{1},\ldots, n_c} \frac{1}{\left\{\prod_{\xi=1}^{c} \left(\sum_{\nu=1}^{\xi}n_\nu+l+m\right)^{r_\xi}\right\} l^{p_1}m^{q_1}}.$$
This implies \eqref{4-5} in the case $a+b=2$.

Next, we consider the general case. 
Mapping each side of \eqref{2-4-2} via the $Z$-map, we have
\begin{equation}
\begin{split}
& Z\left(z_{r_{c}}z_{r_{c-1}} \cdots z_{r_1}(z_{p_a}\cdots z_{p_1}\sh z_{q_b}\cdots z_{q_1})\right) \\
& =\sum_{\tau=0}^{q_b-1}\binom{p_a-1+\tau}{\tau}Z\left(z_{r_{c}}z_{r_{c-1}}\cdots z_{r_1}z_{p_a+\tau}(z_{p_{a-1}}\cdots z_{p_1} \sh z_{q_b-\tau}z_{q_{b-1}}\cdots z_{q_1})\right) \\
&  +\sum_{\tau=0}^{p_a-1}\binom{q_b-1+\tau}{\tau} Z\left(z_{r_{c}}z_{r_{c-1}}\cdots z_{r_1}z_{q_b+\tau}(z_{p_a-\tau}z_{p_{a-1}}\cdots z_{p_1}\sh z_{q_{b-1}}\cdots z_{q_1})\right).
\end{split}
\label{4-2-4}
\end{equation}
Using the assumption of induction, we find that the right-hand side of \eqref{4-2-4} is 
\begin{equation}
\begin{split}
& =\sum_{\tau=0}^{q_b-1}\binom{p_a-1+\tau}{\tau}\sum_{l_1,\ldots,l_{a-1} \atop {m_1,\ldots,m_{b}\atop n_{0}, n_{1},\ldots,n_{c}}} \prod_{\sigma=1}^{c}\left( \sum_{\nu=1}^{\sigma}n_\nu+n_0+\sum_{\mu=1}^{a-1}l_{\mu}+\sum_{\rho=1}^{b}m_\rho\right)^{-r_\sigma}\\
& \quad \quad \times \left( n_0+\sum_{\mu=1}^{a-1}l_{\mu}+\sum_{\rho=1}^{b}m_\rho\right)^{-p_a-\tau}\prod_{\eta=1}^{a-1}\left(\sum_{\mu=1}^{\eta}l_\mu\right)^{-p_\eta} \left(\sum_{\rho=1}^{b}m_\rho\right)^{-q_b+\tau}\prod_{\xi=1}^{b-1}\left(\sum_{\rho=1}^{\xi}m_\rho\right)^{-q_\xi}\\
& \quad +\sum_{\tau=0}^{p_a-1}\binom{q_b-1+\tau}{\tau} \sum_{l_1,\ldots,l_{a} \atop {m_1,\ldots,m_{b-1}\atop n_{0},n_{1},\ldots,n_{c}}} \prod_{\sigma=1}^{c}\left( \sum_{\nu=1}^{\sigma}n_\nu+\sum_{\mu=1}^{a}l_{\mu}+n_0+\sum_{\rho=1}^{b-1}m_\rho\right)^{-r_\sigma}\\
& \quad \quad \times \left( \sum_{\mu=1}^{a}l_{\mu}+n_0+\sum_{\rho=1}^{b-1}m_\rho\right)^{-q_b-\tau}\left(\sum_{\mu=1}^{a}l_\mu\right)^{-p_a+\tau}\prod_{\eta=1}^{a-1}\left(\sum_{\mu=1}^{\eta}l_\mu\right)^{-p_\eta} \prod_{\xi=1}^{b-1}\left(\sum_{\rho=1}^{\xi}m_\rho\right)^{-q_\xi}.
\end{split}
\label{4-2-5}
\end{equation}
On the right-hand side, we rewrite $n_0$ to $l_a$ in the first part, and to $m_b$ in the second part. Then, by using the partial fraction decomposition \eqref{3-3} of with $(X,Y)=\left( \sum_{\mu=1}^{a}l_{\mu},\ \sum_{\rho=1}^{b}m_\rho\right)$ and $(\alpha,\beta)=(p_a,q_b)$, we obtain \eqref{4-5}.
Thus we obtain the proof of the former part. 

The element $z_{r_{c}}z_{r_{c-1}} \cdots z_{r_1}(z_{p_a}\cdots z_{p_1}\sh z_{q_b}\cdots z_{q_1})$ is, if we first apply the step \eqref{2-4-2} of SPP and then operate the $Z$-map, transformed to the right-hand side of \eqref{4-2-4}. On the other hand, the same element is transformed to the right-hand side of \eqref{4-5} via the $Z$-map, and then to \eqref{4-2-5} 
by partial fraction decompositions. 
This implies that the diagram of procedures in Remark \ref{R-2-4} is commutative. Thus we obtain the latter part of our assertion. We complete the proof.
\end{proof}

The explicit correspondence between SPP and the procedure of partial fraction decompositions is given in \eqref{4-2-5}. It is to be emphasized that we can successfully formulate the proof of Theorem \ref{T-2-3}, especially \eqref{4-5} and \eqref{4-2-5}, because we have the framework of zeta-functions of root systems defined by \eqref{2-1}, which is larger than the framework of MZVs.

\begin{proof}[of Theorem \ref{T-2-5}]
In the above proofs of Proposition \ref{P-2-2} and Theorem \ref{T-2-3}, we only used definition \eqref{1-5} and the linearity of the $Z$-map. 
Setting $c=0$ in \eqref{4-5}, we immediately obtain \eqref{2-8}, that is, \eqref{1-9}. 
This is hence a proof of \eqref{1-9} without using Drinfel'd integral expressions \eqref{1-4} for MZVs.
On the other hand, Hoffman proved \eqref{1-7} without using \eqref{1-4} (see \cite[Theorem 4.2]{Ho2}). Therefore the double shuffle relation \eqref{1-10} can be deduced without using \eqref{1-4}. Lastly, in \cite{IKZ}, the regularized $Z$-map $Z^\sh$ is constructed from $Z$ without using \eqref{1-4}. Hence the latter assertion of Theorem \ref{T-2-5} follows.
\end{proof}

\ 

\section{Functional relations including double shuffle relations}\label{sec-5}

As stated in Section \ref{sec-1}, equation \eqref{1-10} is called the double shuffle relation for MZVs. Based on this viewpoint and from the observations in the previous sections, we can give some functional relations for zeta-functions of root systems of $A_N$ type which include double shuffle relations. In this section we explain the 
principle of the method by illustrating the triple and the quadruple cases
in detail. 

From Theorem \ref{T-2-5}, we can obtain \eqref{2-8}, that is \eqref{1-9}, by only using the partial fraction decompositions. Therefore, for example, 
we see that \eqref{1-12} holds for $p,r\in \mathbb{N}$ with $p,r\geq 2$ and $q=s\in \mathbb{C}$ with $\Re(s)>1$, namely
\begin{equation}
\begin{split}
\zeta(p)\zeta(r,s)& = \sum_{i=0}^{r-1}\binom{p-1+i}{i}\zeta({p+i},{r-i},s) \\
& \quad +\sum_{i=0}^{p-1}\binom{r-1+i}{i}\zeta_3({r+i},0,{p-i},0,s,0;A_3).
\end{split}
\label{5-1}
\end{equation}

Similarly, we obtain the following.

\begin{lemma} \label{L-5-1}
For $p,u\in \mathbb{N}$ with $p,u \geq 2$ and $s_1,s_2\in \mathbb{C}$ with $\Re(s_1),\Re(s_2)>1$, 
\begin{equation}
\begin{split}
\zeta(p)\zeta(u,s_2,s_1) & = \sum_{i=0}^{u-1}\binom{p-1+i}{i}\zeta({p+i},{u-i},s_2,s_1) \\
& \quad +\sum_{i=0}^{p-1}\binom{u-1+i}{i}\zeta_4({u+i},0,0,p-i,0,s_2,s_1,0,0,0;A_4)
\end{split}
\label{5-2}
\end{equation}
holds. 
\end{lemma}

\begin{proof}
From \eqref{1-8}, we obtain
\begin{equation*}
\begin{split}
z_p\sh z_uz_rz_q& =\sum_{i=0}^{u-1}\binom{p-1+i}{i}z_{p+i}z_{u-i}z_{r}z_{q} +\sum_{i=0}^{p-1}\binom{u-1+i}{i} z_{u+i}(z_{p-i}\sh z_rz_q). 
\end{split}
\end{equation*}
By \eqref{2-5} in Theorem \ref{T-2-3} and \eqref{2-8} in Theorem \ref{T-2-5}, 
we have
\begin{equation*}
\begin{split}
\zeta(p)\zeta(u,r,q) & =\sum_{i=0}^{u-1}\binom{p-1+i}{i}\zeta(p+i,u-i,r,q) \\
& +\sum_{i=0}^{p-1}\binom{u-1+i}{i}\zeta_4(u+i,0,0,p-i,0,r,q,0,0,0;A_4),
\end{split}
\end{equation*}
where we have only used the partial fraction decompositions. Hence this holds for $(q,r)=(s_1,s_2)\in \mathbb{C}^2$ with $\Re(s_1),\Re(s_2)>1$.
\end{proof}

It is known (see \cite{AET,MatMil,MT1}) that all functions on the both sides of \eqref{5-1} and \eqref{5-2} can be meromorphically continued to the whole complex space. Hence \eqref{5-1} and \eqref{5-2} hold for all $s, s_1,s_2 \in \mathbb{C}$ except for the singularities. 

On the other hand, by considering the harmonic product \eqref{1-6}, we obtain 
\begin{equation}
\begin{split}
& \zeta(p)\zeta(r,s)=\zeta(p,r,s)+\zeta(r,p,s)+\zeta(r,s,p)+\zeta(p+r,s)+\zeta(r,p+s),
\end{split}
\label{5-3}
\end{equation}
and also
\begin{equation}
\begin{split}
& \zeta(p)\zeta(u,s_2,s_1) =\zeta(p,u,s_2,s_1)+\zeta(u,p,s_2,s_1)+\zeta(u,s_2,p,s_1)+\zeta(u,s_2,s_1,p) \\& \quad +\zeta(p+u,s_2,s_1)+\zeta(u,p+s_2,s_1)+\zeta(u,s_2,p+s_1).
\end{split}
\label{5-4}
\end{equation}
We know that \eqref{5-3} and \eqref{5-4} hold for not only integer points but also complex points, from the original meaning of the harmonic product and the meromorphic continuation of multiple zeta-functions (see \cite{AET,AKNagoya,Ess97,MatNagoya,Zhao}). 

Combining \eqref{5-1} - \eqref{5-4}, we obtain the following functional relations which really include double shuffle relations. This implies that the following formulas are answers to the problem proposed by the second-named author (see Section \ref{sec-1}).

\begin{theorem} \label{T-5-1}
For $p_1,p_2 \in \mathbb{N}$ with $p_1,p_2 \geq 2$,
\begin{equation}
\begin{split}
& \sum_{i=0}^{p_2-1}\binom{p_1-1+i}{i}\zeta({p_1+i},{p_2-i},s_1) \\
& \quad +\sum_{i=0}^{p_1-1}\binom{p_2-1+i}{i}\zeta_3({p_2+i},0,{p_1-i},0,s_1,0;A_3)\\
& \quad =\zeta(p_1,p_2,s_1)+\zeta(p_2,p_1,s_1)+\zeta(p_2,s_1,p_1)+\zeta(p_1+p_2,s_1)+\zeta(p_2,p_1+s_1)
\end{split}
\label{5-5}
\end{equation}
and 
\begin{equation}
\begin{split}
& \sum_{i=0}^{p_2-1}\binom{p_1-1+i}{i}\zeta({p_1+i},{p_2-i},{s_2},s_1) \\
& \quad +\sum_{i=0}^{p_1-1}\binom{p_2-1+i}{i}\zeta_4({p_2+i},0,0,p_1-i,0,s_2,s_1,0,0,0;A_4)\\
& \quad =\zeta(p_1,p_2,s_2,s_1)+\zeta(p_2,p_1,s_2,s_1)+\zeta(p_2,s_2,p_1,s_1)+\zeta(p_2,s_2,s_1,p_1) \\
& \quad \quad +\zeta(p_1+p_2,s_2,s_1)+\zeta(p_2,p_1+s_2,s_1)+\zeta(p_2,s_2,p_1+s_1)
\end{split}
\label{5-6}
\end{equation}
hold for all $s_1,s_2 \in \mathbb{C}$ except for the singularities. 
\end{theorem}

\begin{example}
We set $(p_1,p_2)=(2,2)$ in \eqref{5-5}. Then
\begin{equation}
\begin{split}
& \zeta_3({2},0,{2},0,s_1,0;A_3)+2\zeta_3({3},0,{1},0,s_1,0;A_3)\\
& \quad =\zeta(2,2,s_1)-2\zeta({3},{1},s_1) +\zeta(2,s_1,2)+\zeta(4,s_1)+\zeta(2,2+s_1).
\end{split}
\label{5-7}
\end{equation}
From \eqref{3-3}, we obtain
\begin{equation*}
\begin{split}
\zeta_3(a,0,b,0,c,0;A_3)=& \sum_{i=0}^{c-1}\binom{b-1+i}{i}\zeta(a,b+i,c-i) \\
& \ +\sum_{i=0}^{b-1}\binom{c-1+i}{i}\zeta(a,c+i,b-i).
\end{split}
\end{equation*}
Therefore, when $s_1\in \mathbb{N}$, we can see that \eqref{5-7} gives the following double shuffle relation for MZVs: 
\begin{equation}
\begin{split}
& \sum_{i=0}^{k-1}(1+i)\zeta(2,{2+i},k-i)+2\sum_{i=0}^{k-1}\zeta(3,{1+i},k-i) \\
& \quad +\zeta(2,k,2)+k\zeta(2,k+1,1)+2\zeta(3,k,1)\\
& =\zeta(2,2,k)-2\zeta(3,1,k)+\zeta(2,k,2)+\zeta(4,k)+\zeta(2,2+k)\quad (k\in \mathbb{N}).
\end{split}
\label{5-8}
\end{equation}
In particular when $k=1$, we have
$$6\zeta(3,1,1)+\zeta(2,2,1)=\zeta(4,1)+\zeta(2,3).$$

Similarly, combining \eqref{2-4}, \eqref{2-5} and the harmonic product formulas, we can produce functional relations for zeta-functions of the root system of $A_N$ type for any $N$, including double shuffle relations.
\end{example}

\ 

\section{Concluding remarks}\label{sec-6}

In the previous sections we understood that each step of SPP is given by \eqref{2-4-2}. However, the step \eqref{2-4-2} itself is actually obtained by applying repeatedly the fundamental shuffle product relation \eqref{1-8}:
$$u_1w_1\sh u_2w_2=u_1(w_1 \sh u_2w_2)+u_2(u_1w_1 \sh w_2)$$
for $w_1,w_2\in \mathfrak{H}_1$ and $u_1,u_2 \in \{x,y\}$. 
Now we will confirm that even this refined each step of procedures of shuffle products can be realized as the corresponding step of partial fraction decompositions of zeta values of root systems. For this confirmation, we consider \eqref{1-8} for $u_1w_1,u_2w_2\in \mathfrak{H}^{0}$, which implies the case $(u_1,u_2)=(x,x)$, that is, 
\begin{equation}
xw_1\sh xw_2=x(w_1 \sh xw_2)+x(xw_1 \sh w_2). \label{6-3-2}
\end{equation}

First we need to describe $Z(x(w_1\sh w_2))$ in terms of zeta values of root systems. More generally we prepare the following lemma.

\begin{lemma} \label{L-6-2}
Let $a,b,r\in \mathbb{N}$. For $(p_\eta)\in \mathbb{N}^a$, $(q_\xi)\in \mathbb{N}^b$, 
\begin{equation}
\begin{split}
& Z\left(x^r(z_{p_a}z_{p_{a-1}}\cdots z_{p_1}\sh z_{q_b}z_{q_{b-1}}\cdots z_{q_1})\right) =\zeta_{a+b}(\{d_{ij}\};A_{a+b})
\end{split}
\label{6-1}
\end{equation}
holds, where 
\begin{equation}
d_{ij}=
\begin{cases}
r & (i=1;\,j=1) \\
p_{a+b+1-j} & (i=1;\,b+1\leq j \leq a+b) \\
q_{a+b+1-j} & (i=a+1;\,a+1 \leq j \leq a+b)\\
0   & (\text{otherwise}).
\end{cases}
\label{6-2}
\end{equation}
\end{lemma}

\begin{proof}
Using \eqref{2-4} and \eqref{2-5}, and noting $x^rz_p=z_{p+r}$, we obtain 
\begin{equation*}
\begin{split}
& Z\left(x^r(z_{p_a}\cdots z_{p_1}\sh z_{q_b}\cdots z_{q_1})\right) \\
& =\sum_{\tau=0}^{q_b-1}\binom{p_a-1+\tau}{\tau} Z(z_{p_a+\tau+r}(z_{p_{a-1}}\cdots z_{p_1}\sh z_{q_b-\tau}z_{q_{b-1}}\cdots z_{q_1})) \\
& \quad +\sum_{\tau=0}^{p_a-1}\binom{q_b-1+\tau}{\tau} Z(z_{q_b+\tau+r}(z_{p_a-\tau}z_{p_{a-1}}\cdots z_{p_1}\sh z_{q_{b-1}}\cdots z_{q_1}))\\
& =\sum_{\tau=0}^{q_b-1}\binom{p_a-1+\tau}{\tau}\sum_{l_1,\ldots,l_{a-1} \atop {m_1,\ldots,m_{b}\atop n}} \left( n+\sum_{\mu=1}^{a-1}l_{\mu}+\sum_{\rho=1}^{b}m_\rho\right)^{-p_a-\tau-r}\\
& \quad \quad \times \prod_{\eta=1}^{a-1}\left(\sum_{\mu=1}^{\eta}l_\mu\right)^{-p_\eta} \left(\sum_{\rho=1}^{b}m_\rho\right)^{-q_b+\tau}\prod_{\xi=1}^{b-1}\left(\sum_{\rho=1}^{\xi}m_\rho\right)^{-q_\xi}\\
& \quad +\sum_{\tau=0}^{p_a-1}\binom{q_b-1+\tau}{\tau} \sum_{l_1,\ldots,l_{a} \atop {m_1,\ldots,m_{b-1}\atop n}} \left( n+\sum_{\mu=1}^{a}l_{\mu}+\sum_{\rho=1}^{b-1}m_\rho\right)^{-q_b-\tau-r}\\
& \quad \quad \times \left(\sum_{\mu=1}^{a}l_\mu\right)^{-p_a+\tau}\prod_{\eta=1}^{a-1}\left(\sum_{\mu=1}^{\eta}l_\mu\right)^{-p_\eta} \prod_{\xi=1}^{b-1}\left(\sum_{\rho=1}^{\xi}m_\rho\right)^{-q_\xi}.
\end{split}
\end{equation*}
On the right-hand side, we rewrite $n$ to $l_a$ in the first part, and to $m_b$ in the second part. Then the right-hand side is equal to
\begin{equation*}
\begin{split}
& \sum_{l_1,\ldots,l_{a} \atop {m_1,\ldots,m_{b}}} \left(\sum_{\mu=1}^{a}l_{\mu}+\sum_{\rho=1}^{b}m_\rho\right)^{-r}\prod_{\eta=1}^{a-1}\left(\sum_{\mu=1}^{\eta}l_\mu\right)^{-p_\eta} \prod_{\xi=1}^{b-1}\left(\sum_{\rho=1}^{\xi}m_\rho\right)^{-q_\xi}\\
& \quad \times \bigg\{ \sum_{\tau=0}^{q_b-1}\binom{p_a-1+\tau}{\tau} \left(\sum_{\mu=1}^{a}l_{\mu}+\sum_{\rho=1}^{b}m_\rho\right)^{-p_a-\tau}\left(\sum_{\rho=1}^{b}m_\rho\right)^{-q_b+\tau} \\
& \quad\quad + \sum_{\tau=0}^{p_a-1}\binom{q_b-1+\tau}{\tau} \left(\sum_{\mu=1}^{a}l_{\mu}+\sum_{\rho=1}^{b}m_\rho\right)^{-q_b-\tau}\left(\sum_{\mu=1}^{a}l_\mu\right)^{-p_a+\tau}\bigg\}.
\end{split}
\end{equation*}
Using the partial fraction decomposition \eqref{3-3} with $(X,Y)=\left( \sum_{\mu=1}^{a}l_{\mu},\ \sum_{\rho=1}^{b}m_\rho\right)$ and $(\alpha,\beta)=(p_a,q_b)$, 
we have
\begin{equation}
\begin{split}
& Z\left(x^r(z_{p_a}z_{p_{a-1}}\cdots z_{p_1}\sh z_{q_b}z_{q_{b-1}}\cdots z_{q_1})\right) \\
&\quad =\sum_{l_1,l_2,\ldots ,l_a \in \mathbb{N} \atop m_1,m_{2},\ldots,m_b \in \mathbb{N}} \left( \sum_{\mu=1}^{a}l_\mu+\sum_{\rho=1}^{b}m_\rho\right)^{-r}\prod_{\eta=1}^{a}\left(\sum_{\mu=1}^{\eta}l_\mu\right)^{-p_\eta} \prod_{\xi=1}^{b}\left(\sum_{\rho=1}^{\xi}m_\rho\right)^{-q_\xi},
\end{split}
\label{6-3}
\end{equation}
which coincides with $\zeta_{a+b}(\{d_{ij}\};A_{a+b})$ for $\{d_{ij}\}$ determined by \eqref{6-2}. This completes the proof.
\end{proof}

From \eqref{6-3-2} and the linearity of $Z$, we obtain
\begin{equation}
Z(xw_1\sh xw_2)=Z(x(w_1 \sh xw_2))+Z(x(xw_1 \sh w_2)).  \label{6-4}
\end{equation}
On the other hand, in Theorem \ref{T-2-5}, we have already proved that 
\begin{equation}
\begin{split}
& Z\left(xz_{p_a}z_{p_{a-1}}\cdots z_{p_1}\sh xz_{q_b}z_{q_{b-1}}\cdots z_{q_1}\right) =Z\left(xz_{p_a}z_{p_{a-1}}\cdots z_{p_1}\right)Z\left(xz_{q_b}z_{q_{b-1}}\cdots z_{q_1}\right) \label{6-6}
\end{split}
\end{equation}
without using Drinfel'd integral expressions for multiple zeta values. 
By the simple partial fraction decomposition
$$\frac{1}{X^{\alpha+1} Y^{\beta+1}}=\frac{1}{(X+Y)X^{\alpha}Y^{\beta+1}}+\frac{1}{(X+Y)X^{\alpha+1}Y^{\beta}}$$
with $(X,Y)=\left(\sum_{\mu=1}^{a}l_\mu, \sum_{\rho=1}^{b}m_\rho\right)$, we obtain
\begin{equation}
\begin{split}
& Z\left(xz_{p_a}z_{p_{a-1}}\cdots z_{p_1}\right)Z\left(xz_{q_b}z_{q_{b-1}}\cdots z_{q_1}\right) \\
& =\sum_{l_1,l_2,\ldots ,l_a \atop m_1,m_{2},\ldots,m_b} \prod_{\eta=1}^{a-1}\left(\sum_{\mu=1}^{\eta}l_\mu\right)^{-p_\eta} \left(\sum_{\mu=1}^{a}l_\mu\right)^{-p_a-1}\prod_{\xi=1}^{b-1}\left(\sum_{\rho=1}^{\xi}m_\rho\right)^{-q_\xi}\left(\sum_{\rho=1}^{b}m_\rho\right)^{-q_b-1}\\
& =\sum_{l_1,l_2,\ldots ,l_a \atop m_1,m_{2},\ldots,m_b} \bigg\{ \left( \sum_{\mu=1}^{a}l_\mu+\sum_{\rho=1}^{b}m_\rho\right)^{-1}\prod_{\eta=1}^{a}\left(\sum_{\mu=1}^{\eta}l_\mu\right)^{-p_\eta} \left(\sum_{\rho=1}^{b}m_\rho\right)^{-q_b-1}\prod_{\xi=1}^{b-1}\left(\sum_{\rho=1}^{\xi}m_\rho\right)^{-q_\xi}\\
& \ \ +\sum_{l_1,l_2,\ldots ,l_a \atop m_1,m_{2},\ldots,m_b} \left( \sum_{\mu=1}^{a}l_\mu+\sum_{\rho=1}^{b}m_\rho\right)^{-1}\left(\sum_{\mu=1}^{a}l_\mu\right)^{-p_a-1}\prod_{\eta=1}^{a-1}\left(\sum_{\mu=1}^{\eta}l_\mu\right)^{-p_\eta} \prod_{\xi=1}^{b}\left(\sum_{\rho=1}^{\xi}m_\rho\right)^{-q_\xi}\bigg\}\\
& =Z\left(x(z_{p_a}z_{p_{a-1}}\cdots z_{p_1}\sh xz_{q_b}z_{q_{b-1}}\cdots z_{q_1})\right) +Z\left(x(xz_{p_a}z_{p_{a-1}}\cdots z_{p_1}\sh z_{q_b}z_{q_{b-1}}\cdots z_{q_1})\right),
\end{split}
\label{6-5}
\end{equation}
where the last equality follows from Lemma \ref{L-6-2}. 
The combination of \eqref{6-6} 
and \eqref{6-5} implies that the left-hand side of \eqref{6-4} can be, via $Z\left(xz_{p_a}z_{p_{a-1}}\cdots z_{p_1}\right)Z\left(xz_{q_b}z_{q_{b-1}}\cdots z_{q_1}\right)$, transformed to the right-hand side of \eqref{6-4}. This shows that the diagram in Remark \ref{R-2-4} is commutative in the sense that SPP implies the step given by \eqref{6-3-2}. Therefore each step of SPP given by \eqref{6-3-2} can be realized as the corresponding step of partial fraction decompositions for zeta values of root systems.

\ 

\noindent
{\bf Acknowledgements.}\\
A part of this work was done during the authors' stay at Universit\"at W\"urzburg in the winter semester 2008/09 (November for H.T.). The authors express their sincere gratitude to Professor J\"orn Steuding and Dr. Rasa Steuding for the invitation and hospitality.


\begin{thebibliography}{0}
\bibitem{AET}
{Akiyama, S., Egami, S., Tanigawa, Y.}: 
Analytic continuation of multiple
zeta functions and their values at non-positive integers. 
{Acta Arithmetica} 98, 107--116 (2001)

\bibitem{AKNagoya}
{Arakawa, T., Kaneko, M.}: 
Multiple zeta values, poly-Bernoulli numbers and related zeta functions.
{Nagoya Mathematical Journal} {153}, 189--209 (1999)  

\bibitem{AK}
{Arakawa, T., Kaneko, M.}: 
On multiple $L$-values. 
{Journal of the Mathematical Society of Japan} {56}, 967--991 (2004) 

\bibitem{BBB0}
Borwein, J. M., Bradley, D. M., Broadhurst, D. J., Lisonek, P.: 
Combinatorial aspects of multiple zeta values. 
{Electronic Journal of Combinatorics} {5}, R38 (1998) 

\bibitem{BBB1}
Borwein, J. M., Bradley, D. M., Broadhurst, D. J., Lisonek, P.: 
Special values of multidimensional polylogarithms.
{Transactions of the American Mathematical Society} {353}, 907--941 (2001) 

\bibitem{BB0}
{Bowman, D., Bradley, D. M.}: 
Multiple polylogarithms: a brief survey. 
In {Conference on $q$-Series with Applications to Combinatorics,
Number Theory, and Physics} (Urbana, IL, 2000), Contemp. Math.
291, B. C. Berndt and K. Ono (eds.), Amer. Math. Soc., Providence,
RI, pp. 71--92 (2001)

\bibitem{BB1}
{Bowman, D., Bradley, D. M.}: 
The algebra and combinatorics of shuffles and multiple zeta values. 
{Journal of Combinatorial Theory, Ser. A} {97}, 43--61 (2002) 

\bibitem{BB2}
{Bowman, D., Bradley, D. M.}: 
Resolution of some open problems concerning multiple zeta evaluations of arbitrary depth. 
{Compositio Mathematica} {139}, 85--100 (2003) 

\bibitem{Brad}
{Bradley, D. M.}: 
Partition identities for the multiple zeta function. 
In {Zeta Functions, Topology and Quantum Physics}, Developments in Mathematics 14, T. Aoki et al. (eds.), Springer, New York, pp. 19--29 (2005)

\bibitem{Dri}
Drinfel'd, V. G.: 
On quasitriangular quasi-Hopf algebras and a group closely related with {Gal}$(\overline{\bf Q}/{\bf Q})$.(Russian) {Algebra i Analiz} 2 (1990), no. 4, 149--181; translation in {Leningrad Mathematical Journal} 2, no. 4, 829--860 (1991)

\bibitem{Ess97}
{Essouabri, D.}: 
Singularit{\'e}s des s{\'e}ries de Dirichlet
associ{\'e}es {\`a} des polyn{\^o}mes de plesieurs variables et applications
en th{\'e}orie analytique des nombres. 
{Annales de L'Institut Fourier (Grenoble)} {47}, 429--483 (1997)   

\bibitem{Gon}
Goncharov, A. B.: 
Periods and mixed motives. 
preprint (2002), arXiv:math/0202154.

\bibitem{HWZ}
{Huard, J. G., Williams, K. S., Zhang, N.-Y.}: 
On Tornheim's double series. 
{Acta Arithmatica} {75}, 105--117 (1996)

\bibitem{Ho}
{Hoffman, M. E.}: 
Multiple harmonic series. 
{Pacific Journal of Mathematics} {152}, 275--290 (1992)

\bibitem{Ho2}
{Hoffman, M. E.}: 
The algebra of multiple harmonic series. 
{Journal of Algebra} {194}, 477--495 (1997)

\bibitem{Ho3}
{Hoffman, M. E.}: 
Algebraic aspects of multiple zeta values. 
In {Zeta Functions, Topology and Quantum Physics}, Developments in Mathematics 14, T. Aoki et al. (eds.), Springer, New York, pp. 51-74 (2005) 

\bibitem{Ho-Oh}
Hoffman, M. E., Ohno, Y.: 
Relations of multiple zeta values and their algebraic expression. 
{Journal of Algebra} {262}, 332--347 (2003)

\bibitem{IKZ}
{Ihara, K., Kaneko, M., Zagier, D.}: 
Derivation and double shuffle relations for multiple zeta values. 
{Compositio Mathematica} {142}, 307--338 (2006)

\bibitem{Kaneko}
Kaneko, M.: 
Multiple zeta values. 
{Sugaku Expositions} {18}, 221--232 (2005)

\bibitem{KMT1}
{Komori, Y., Matsumoto, K., Tsumura, H.}: 
Zeta-functions of root systems. 
In {The Conference on $L$-functions} (Fukuoka 2006), L. Weng and M. Kaneko (eds.), World Sci. Publ., Hackensack, NJ, pp. 115--140 (2007) 

\bibitem{KMT2}
{Komori, Y., Matsumoto, K., Tsumura, H.}: 
Zeta and $L$-functions and Bernoulli polynomials of root systems. 
{Proceedings of the Japan Academy, Ser. A} {84}, 57--62 (2008)

\bibitem{KMT-L}
{Komori, Y., Matsumoto, K., Tsumura, H.}: 
On multiple Bernoulli polynomials and multiple $L$-functions of root systems. 
{Proceedings of the London Mathematical Society}, to appear.

\bibitem{KMT3}
{Komori, Y., Matsumoto, K., Tsumura, H.}: 
On Witten multiple zeta-functions associated with semisimple Lie algebras II. 
{J. Math. Soc. Japan}, to appear. 

\bibitem{KMT4}
{Komori, Y., Matsumoto, K., Tsumura, H.}: 
On Witten multiple zeta-functions associated with semisimple Lie algebras III. 
Preprint, arXiv:math/0907.0955. 

\bibitem{Mark}
Markett, C.: 
Triple sums and the Riemann zeta function.
{Journal of Number Theory} {48}, 113--132 (1994)

\bibitem{MatMil}
{Matsumoto, K.}: 
On the analytic continuation of various multiple zeta-functions. 
In {Number Theory for the Millennium II}, M. A. Bennett et al. (eds.), A K Peters, pp. 417--440 (2002) 

\bibitem{MatNagoya}
{Matsumoto, K.}: 
Asymptotic expansions of double zeta-functions of Barnes, of Shintani, and Eisenstein series. 
{Nagoya Mathematical Journal} {172}, 59--102 (2003)

\bibitem{MatBonn}
{Matsumoto, K.}: 
On Mordell-Tornheim and other multiple zeta-functions. 
In {Proceedings of the Session in Analytic Number Theory and Diophantine Equations} (Bonn, 2002), Bonner Math. Schriften 360, D.R.Heath-Brown, B.Z.Moroz (eds.), Bonn, n.25, 17pp (2003)

\bibitem{MatXi}
{Matsumoto, K.}: 
Analytic properties of multiple zeta-functions in several variables. 
In {Number Theory: Tradition and Modernization}, Proceedings of the 3rd China-Japan Seminar, W. Zhang and Y. Tanigawa (eds.), Springer, pp. 153-173 (2006) 

\bibitem{MT1}
{Matsumoto, K., Tsumura, H.}: 
On Witten multiple zeta-functions associated with semisimple Lie algebras I. 
{Annales de L'Institut Fourier (Grenoble)} {56}, 1457--1504 (2006)

\bibitem{MT2}
{Matsumoto, K., Tsumura, H.}: 
A new method of producing functional relations among multiple zeta-functions.
{Quarterly Journal of Mathematics (Oxford)} {59}, 55--83 (2008)

\bibitem{Mo}
{Mordell, L. J. }: 
On the evaluation of some multiple series. 
{Journal of the London Mathematical Society} {33}, 368--371 (1958)

\bibitem{Ohno}
{Ohno, Y.}: 
A generalization of the duality and sum formulas on the multiple zeta values.
{Journal of Number Theory} {74}, 39-43 (1999)

\bibitem{O-Z}
Ohno, Y., Zagier, D.: 
Multiple zeta values of fixed weight, depth, and height. 
{Indagationes Mathematicae (N. S.)} {12}, 483--487 (2001)

\bibitem{Reu1}
{Reutenauer, C.}: 
The shuffle algebra on the factors of a word is free. 
{Journal of Combinatorial Theory, Ser. A} {38}, 48--57 (1985)

\bibitem{Reu2}
{Reutenauer, C.}: 
{Free Lie Algebra}. 
Oxford Science Publications, 1993.

\bibitem{Tera}
Terasoma, T.: 
Mixed Tate motives and multiple zeta values.
{Inventiones Mathematicae} {149}, 339--369 (2002)

\bibitem{To}
{Tornheim, L.}: 
Harmonic double series. 
{American Journal of Mathematics} {72}, 303--314 (1950)

\bibitem{TsCa}
{Tsumura, H.}: 
On functional relations between the Mordell-Tornheim double zeta functions and the Riemann zeta function. 
{Mathematical Proceedings of the Cambridge Philosophical Society} {142}, 395--405 (2007)

\bibitem{Wi}
{Witten, E.}: 
On quantum gauge theories in two dimensions.
{Communications in Mathematical Physics} {141}, 153--209 (1991)

\bibitem {Za} 
{Zagier, D.}: 
Values of zeta functions and their applications.
In {First European Congress of Mathematics}, Vol. II (Paris, 1992): 497--512, Progress in Mathematics 120, Birkh\"auser, Basel, 1994.

\bibitem{Zhao}
{Zhao, J.}: 
Analytic continuation of multiple zeta functions. 
{Proceedings of the American Mathematical Society} {128}, 1275--1283 (2000)

\end{thebibliography}


\end{document}